\documentclass[11pt]{amsart}
\usepackage[utf8]{inputenc}
\usepackage{amsmath,amsfonts,amssymb,amscd}
\usepackage{calligra}
\usepackage[hidelinks]{hyperref}
\newcommand{\m}{{\mathcal{M}}}

\newcommand {\OO}{{\mathcal O}}

\newcommand{\C}{\mathbb C}

\newcommand{\MM}{{\calligra M~}}
\newcommand{\I}{\mathcal{I}}
\newtheorem{theorem}{Theorem}

\newtheorem{prop}[theorem]{Proposition}

\theoremstyle{definition}
\newtheorem{defn}[theorem]{Definition}

\begin{document}
\title{Nonprojective \MM-manifolds}
\author{Blake J. Boudreaux}
\address{Department of Mathematical Sciences, University of Arkansas, Fayetteville, AR 72701, USA}
\email{bb225@uark.edu}
\author{Rasul Shafikov}
\address{Department of Mathematics, University of Western Ontario, London, Ontario N6A 5B7, Canada}
\email{shafikov@uwo.ca}
\date{\today}
\begin{abstract}
    In~\cite{BoSh26}, the authors introduced, by analogy with Stein manifolds, a class of complex manifolds called {\calligra M~}-manifolds. This class contains all complex projective manifolds. Moreover, every compact {\calligra M~}-manifold is Moishezon. In this note, we construct examples showing that the class of compact {\calligra M~}-manifolds sits properly between the classes of projective and Moishezon manifolds.
\end{abstract}

\maketitle

Holomorphic convexity and Stein manifolds play a key role in the development of the holomorphic function theory of several complex variables. In~\cite{BoSh25, BoSh26}, the authors began the development of a meromorphic analogue of these two concepts.

Namely, given an arbitrary complex manifold $X$ and a compact set $K\subseteq X$, we define the \textit{meromorphically convex hull of $K$}, denoted $\widehat{K}_X$, by
\begin{equation*}
    \widehat{K}_X=\left\{z\in X\,:\,|f(z)|\leq \|f\|_{K} \text{ for every }f\in\m(X)\cap\OO(K\cup \{z\})\right\} ,
\end{equation*}
where $\m(X)$ is the space of meromorphic functions on $X$.
We say that the complex manifold $X$ is \textit{meromorphically convex} if $\widehat{K}_X$ is compact in $X$ for every compact set $K\subseteq X$. Proposition~4(i) of~\cite{BoSh26} shows that $\widehat{K}_X$ is a closed set, so it follows that every compact complex manifold $X$ is meromorphically convex.

Equipped with the notion of meromorphic convexity, we define the following meromorphic analogue of a Stein manifold, known as an \MM-manifold.
\begin{defn}\label{Mmanifold}
    A complex manifold $X$ is called an \MM-manifold if the following conditions are satisfied:
    \begin{enumerate}
	\item[(a)] $X$ is meromorphically convex, i.e., $\widehat K_{X}$ is a compact subset of $X$ for any compact set $K\subseteq X$;
	\item[(b)] $\m(X)$ separates points, i.e., for any pair of distinct points $p,q\in X$ there exists a meromorphic function $f$ on $X$ such that $f$ is holomorphic near $p$ and $q$, and $f (p)\neq f (q)$;
	\item[(c)] Existence of local coordinates: for any point $p\in X$ there exists a neighborhood $U$ of $p$ and meromorphic functions $f_{1},\ldots,f_{n}$ such that $\left\{f_{1}|_U,\ldots,f_{n}|_U\right\}$ forms a local holomorphic coordinate system on $U$.
    \end{enumerate}
\end{defn}
Basic examples of \MM-manifolds are provided by Stein and projective manifolds, together with blowups of each and Cartesian products of a Stein manifold with a projective manifold~\cite[\textsection 3]{BoSh26}.

A related class is that of Moishezon manifolds. These are compact manifolds for which $\m(X)$, viewed as a field extension over $\C$, has maximal transcendence degree ($=\dim_{\C}X$). This class of manifolds includes all projective manifolds. On the other hand, compact \MM-manifolds are Moishezon~\cite[Proposition~11]{BoSh26}. This gives the following chain of inclusions:
\begin{equation}\label{inclusions}
    \text{projective manifolds}\subseteq \text{compact \MM-manifolds}\subseteq \text{Moishezon manifolds}.
\end{equation}

The primary purpose of this note is to observe the following.
\begin{theorem}\label{proper}
    Both inclusions in~\eqref{inclusions} above are proper.
\end{theorem}

To prove Theorem~\ref{proper}, we will furnish a Moishezon manifold $A$ that is not an \MM-manifold and a compact \MM-manifold $B$ which is not projective. Interestingly, both examples originate in constructions of Hironaka~\cite{Hi60, Hi62}, which were initially intended to show that a complete algebraic manifold need not be projective, and that a Moishezon manifold need not be algebraic. We recall both constructions below, following Appendix B of Hartshorne~\cite{Ha77}.

We first recall the construction of $A$. Let $c$ be an irreducible curve in a projective threefold $X$, which is nonsingular except for a single self-intersection point $p$ at which $c$ intersects itself nontangentially. Locally near $p$, the curve $c$ has two smooth branches with distinct tangent directions. Near $p$, blow up one branch, then (the strict transform of) the other; away from $p$, simply blow up $c$. These local constructions agree off of $p$ and hence glue together to a compact complex manifold $A$ with a proper modification $\pi:A\to X$. In particular, $A$ is Moishezon, since $\m(X)\cong\m(A)$. The fiber $\pi^{-1}(p)$ consists of two rational curves, and the component $c_{0}$ arising from the branch blown up first is null-homologous, i.e., $[c_{0}]=0$ in $H_{2}(A;\mathbb Z)$~\cite[Example B.3.4.2]{Ha77}. By Proposition~\ref{torsionProp} below, $A$ is not an \MM-manifold.

We now turn to the construction of $B$. Let $X$ be any nonsingular projective algebraic complex manifold of three complex dimensions. Take any two nonsingular curves $c_1,c_2\subset X$ which intersect nontangentially at precisely two distinct points $p,q$. For $k=1,2$, let $\sigma_k:X_k\to X$ denote the blow-up of $X$ along $c_k$, and for distinct $j,k\in \{1,2\}$ let $\sigma_{jk}:X_{jk}\to X_k$ denote the blow-up of $X_k$ along the strict transform of $c_j$. We write $\pi_{jk}=\sigma_k\circ\sigma_{jk}:X_{jk}\to X$ for the resulting proper modification. We note that $X_{1},X_2,X_{12}$, and $X_{21}$ are projective algebraic manifolds~\cite[Proposition II.7.16]{Ha77}. Next, $X_{12}$ and $X_{21}$ coincide away from the inverse images of $p$ and $q$, so we can glue $X_{12}$ and $X_{21}$ together there to obtain a nonsingular complete algebraic variety $B$ along with a proper modification $\pi: B\to X$ such that
\begin{equation}\label{gluing}
  B\setminus\pi^{-1}(q)\cong X_{12}\setminus\pi_{12}^{-1}(q)
  \quad\text{and}\quad
  B\setminus\pi^{-1}(p)\cong X_{21}\setminus\pi_{21}^{-1}(p),
\end{equation}
where under the first identification $\pi$ agrees with $\pi_{12}$, and under the second $\pi$ agrees with $\pi_{21}$. The variety $B$ is not projective: if $\ell_0$ and $m_0'$ denote suitable components of the fibres $\pi^{-1}(p)$ and $\pi^{-1}(q)$, then $[\ell_0]+[m_0']=0$ in $H_2(B;\mathbb Z)$, which is impossible on a projective manifold~\cite[Example B.3.4.1]{Ha77}.

\begin{prop}\label{Bprop}
    The nonsingular complete algebraic variety $B$ is an \MM-manifold.
\end{prop}
\begin{proof}
    Since $B$ is compact, condition (a) of Definition~\ref{Mmanifold} is immediate, as noted above. It therefore remains only to verify conditions (b) and (c). For brevity, set $B_p=B\setminus\pi^{-1}(q)$ and $B_q=B\setminus\pi^{-1}(p)$, so that $B=B_p\cup B_q$, and identify $B_p$ and $B_q$ with open subsets of $X_{12}$ and $X_{21}$, respectively, via~\eqref{gluing}.

    We first make the following observation: for every $f\in\m(X_{12})$ there exists $h\in\m(B)$ such that $h=f$ on $B_p$. Indeed, the proper modification $\pi_{12}:X_{12}\to X$ induces an isomorphism $\m(X)\to\m(X_{12})$, $g\mapsto g\circ\pi_{12}$, so there exists a unique $g\in\m(X)$ with $g\circ\pi_{12}=f$. Define $h=g\circ\pi\in\m(B)$. Since $\pi$ agrees with $\pi_{12}$ on $B_p$, we have $h=g\circ\pi_{12}=f$ on $B_p$. The same argument, with $X_{21}$ and $\pi_{21}$ in place of $X_{12}$ and $\pi_{12}$, shows that every $f\in\m(X_{21})$ extends to some $h\in\m(B)$ with $h=f$ on $B_q$.

    We now verify condition (b). Let $a,b\in B$ be distinct points. Suppose first that $\pi (a)\neq \pi (b)$. The manifold $X$ is projective and hence an \MM-manifold~\cite[\textsection 3]{BoSh26}, so there exists $g\in \m(X)$ that is holomorphic near $\pi(a)$ and $\pi(b)$ with $g(\pi(a))\neq g(\pi(b))$. It follows that $g\circ\pi \in\m(B)$ is holomorphic near $a$ and $b$ and separates $a$ and $b$.

    Suppose now that $\pi(a)=\pi(b)$. Then $a$ and $b$ lie in the same fibre of $\pi$, and since $p\neq q$, either $a,b\in B_p$ or $a,b\in B_q$. For definiteness, suppose that $a,b\in B_p$; the other case is handled identically. Since $X_{12}$ is projective and hence an \MM-manifold, there exists $f\in\m(X_{12})$ which is holomorphic near $a$ and $b$ with $f(a)\neq f(b)$. By the observation above, there exists $h\in\m(B)$ with $h=f$ on $B_p$. Then $h$ is holomorphic near $a$ and $b$ and $h(a)\neq h(b)$, so $h$ separates $a$ and $b$.

    Finally, we verify condition (c). Fix a point $a\in B$. Then $a\in B_p$ or $a\in B_q$, and for definiteness we suppose that $a\in B_p$. Since $X_{12}$ is an \MM-manifold, there exist $f_{1},f_{2},f_{3}\in\m(X_{12})$ which form a local holomorphic coordinate system near $a$ on $X_{12}$. By the observation above, there exist $h_{1},h_{2},h_{3}\in\m(B)$ with $h_j=f_j$ on $B_p$ for $j=1,2,3$. In particular, $h_j$ agrees with $f_j$ near $a$, and so $h_{1},h_{2},h_{3}$ form a local holomorphic coordinate system on $B$ near $a$. This completes the proof.
\end{proof}

\begin{prop}\label{torsionProp}
    If a compact complex manifold $X$ contains an irreducible compact curve $c$ whose fundamental class $[c]\in H_{2}(X;\mathbb Z)$ is a torsion element, then $X$ is not an \MM-manifold.
\end{prop}
\begin{proof}[Proof (cf.~{\cite[pp.~8--9]{BoSh26}})]
    Let $\nu:\tilde{c}\to c$ be the normalization of $c$, so that $\tilde c$ is a compact connected Riemann surface. We claim that every $f\in\m(X)$ is either constant on $c$ or indeterminate along all of $c$, i.e., $c\subseteq\I(f)$.

    Seeking a contradiction, suppose that $f\in\m(X)$ satisfies $c\not\subseteq\I(f)$ and $f|_{c}$ is nonconstant. Then $f\circ\nu$ is a nonconstant meromorphic function on $\tilde c$, i.e., a nonconstant holomorphic map $\tilde c\to\mathbb{CP}^1$. Choose a point $\lambda$ in the image of $f\circ\nu$, and let $D_{\lambda}$ be the effective divisor on $X$ given by the closure of $f^{-1}(\lambda)$ in $X$ (the zero divisor of $f-\lambda$ if $\lambda\neq\infty$, and the polar divisor of $f$ if $\lambda=\infty$). Then $D_\lambda$ does not contain $c$, since $f|_c$ is nonconstant, and $D_{\lambda}\cap c\neq\emptyset$ by the choice of $\lambda$ (note that every point of $\I(f)$ lies on $D_{\lambda}$). Hence $D_{\lambda}$ meets $c$ in a nonempty finite set, and therefore
    \begin{equation*}
        \langle c_{1}(\OO_{X}(D_{\lambda})),[c]\rangle=\deg\, \nu^{*}\OO_{X}(D_{\lambda})>0,
    \end{equation*}
    the degree being the number of intersection points of $D_\lambda$ and $c$ counted with multiplicities. On the other hand, $\langle c_{1}(\OO_{X}(D_{\lambda})),[c]\rangle=0$ since $[c]$ is a torsion element of $H_2(X;\mathbb Z)$. This contradiction proves the claim.

    It follows that no $f\in\m(X)$ can be holomorphic near two distinct points $a,b\in c$ and satisfy $f(a)\neq f(b)$. Thus $\m(X)$ does not separate points of $c$, condition (b) of Definition~\ref{Mmanifold} fails, and $X$ is not an \MM-manifold.
\end{proof}

Since $A$ is Moishezon but not an \MM-manifold, and $B$ is a nonprojective compact \MM-manifold, this completes the proof of Theorem~\ref{proper}.

\bigskip

We note that the manifold $B$ is an example of an \MM-manifold that is not K\"ahler.
We do not know whether the converse to Proposition~\ref{torsionProp} holds for Moishezon manifolds; i.e., whether every compact Moishezon manifold that is not an \MM-manifold contains an irreducible closed curve whose fundamental class in $H_2(X;\mathbb Z)$ is a torsion element. This remains to be an open question.

\end{document}